\documentclass{amsart}
\usepackage{amsfonts,amsthm,amsmath,amssymb,amscd,mathrsfs}
\usepackage{graphics}
\usepackage{indentfirst}
\usepackage{cite}
\usepackage{latexsym}
\usepackage[dvips]{epsfig}
\usepackage[utf8]{inputenc}
\usepackage{color}
\usepackage{esint}
\usepackage[colorlinks=true,linkcolor=blue]{hyperref}
\usepackage{verbatim}

\newtheorem{theorem}{Theorem}[section]
\newtheorem{lemma}[theorem]{Lemma}

\theoremstyle{definition}
\newtheorem{definition}[theorem]{Definition}

\theoremstyle{remark}
\newtheorem{remark}[theorem]{Remark}

\numberwithin{equation}{section}

\DeclareMathOperator{\divg}{div}

\begin{document}

\title[Steady viscous channel Flows]{On an inhomogeneous slip-inflow boundary value problems for a steady viscous compressible channel flows }

\author{Wengang Yang}
\address{School of Mathematics and Statistics, Wuhan University, Wuhan, Hubei Province, 430072, People's Republic of China.}
\email{yangwg@whu.edu.cn}
\begin{abstract}
We prove the existence and uniqueness of strong solutions to the steady isentropic 
compressible Navier-Stokes equations with inflow boundary conditions for density and 
mixed boundary conditions for the velocity around a shear flow. In particular, the 
Dirichlet boundary condition on inflow and outflow part of the boundary while the 
full Navier  boundary conditions on the wall $\Gamma_{0}$ for the velocity filed are 
considered. For our result, there are no  restrictions on the  amplitude of friction 
coefficients $\alpha$, and the appropriately large  hypothesis for the viscous 
coefficients $\mu$ is enough. One of the  substantial ingredients of our proof is an 
elegant transformation induced by the flow field. With the help of this 
transformation,  we can overcome the 
difficulties caused by the hyperbolicity of continuity equation, establish the a 
priori estimates for a linearized  system and apply the fixed point arguments.
\end{abstract}
%  Address of record for the research reported here
%%\address{Department of Mathematics, name of school, city, location}
%%\curraddr{Department of Mathematics and Statistics,
%%Case Western Reserve University, Cleveland, Ohio 43403}
%%\email{xyz@math.university.edu}


%    General info
%\subjclass[2000]{Primary 54C40, 14E20; Secondary 46E25, 20C20}

%\date{January 1, 2001 and, in revised form, June 22, 2001.}


%%\begin{abstract}
%%this is an abstract.
%%\begin{description}
%% \item[Usage]
 %%\item[Structure]
%% You may use the \texttt{description} environment to structure your abstract;
%% use the optional argument of the \verb+\item+ command to give the category of each item.
%% \end{description}
%%\end{abstract}

%\titlerunning{short title}
\maketitle

\section{introduction and main results}
In this paper, we consider the steady isentropic compressible Navier-Stokes 
equations with inflow boundary condition in a two-dimensional tube $\Omega=(0,  
1)\times(0, 1)$ near a shear flow. It is well known that the Navier–Stokes 
equations for a steady isentropic compressible viscous flow is a mixed
system of hyperbolic-elliptic type. The momentum equations are an elliptic system in 
the velocity, while the continuity equation is hyperbolic in the density. Therefore, 
it is necessary to prescribe the density on the part of inflow boundary $(\mathbf{u} 
\cdot \mathbf{n} <0)$\, where $\mathbf{n}$ is the outward unit normal to the part of 
inflow boundary.  The inflow boundary value problem considered in this article reads 
as follows:
\begin{equation}\label{S-NS}
\begin{aligned}
\operatorname{div}(\rho \mathbf{u}) &=0 & &\text { in } \Omega, \\
\operatorname{div}(\rho \mathbf{u} \otimes \mathbf{u})-\mu \Delta \mathbf{u}-(\mu+\nu) \nabla \operatorname{div}{\bf u}+\nabla P &=0 & &\text { in } \Omega, \\
\mathbf{u} &=\mathbf{u}_{0} & &\text { on } \Gamma_{\text {in }} \cup \Gamma_{\text {out }}, \\
\rho &=\rho_{\text {in }} & &\text { on } \Gamma_{\text {in }}, \\
\mathbf{u} \cdot \mathbf{n} &=0 & &\text { on } \Gamma_{0}, \\
2\mu\mathbf{n}\cdot D({\bf u})\cdot \tau +\alpha{\bf u}\cdot \tau &= b & &\text { on } \Gamma_{0},
\end{aligned}
\end{equation}
where ${\bf u}:\mathbb{R}^2\rightarrow\mathbb{R}^2$ is the unknown velocity filed of the fluid and  $\rho:\mathbb{R}^2\rightarrow\mathbb{R}$ is the unknown density. The $D(\bf{ u})$ denotes the symmetric part of the velocity gradient, more precisely,
\begin{equation*}
  D({\bf u})=\frac{1}{2}(\nabla {\bf u}+(\nabla{\bf u})^\mathrm{T}).
\end{equation*}
The constant viscosity coefficients $\mu$ and $\nu$ satisfy $\mu > 0,\nu+\mu \geq 0$ and $\alpha\geq 0$ is a friction coefficient. The pressure P is usually determined through the equation of states. In the case of isentropic flows, the pressure P is given by $P(\rho)=\rho^{\gamma}$  with $\gamma>1$ being the heat ratio. Next, ${\bf n}$ and $\tau$ are outer normal and tangent vectors to $\partial\Omega$. The boundary $\partial\Omega$ is naturally divided into the inflow part $\Gamma_{in}$, the outflow part $\Gamma_{out}$, and the impermeable wall $\Gamma_0$, more precisely,
\begin{equation*}
  \begin{aligned}
    \Gamma_{in}&=\{x_1=0,0\leq x_2\leq 1\},\\
    \Gamma_{out}&=\{x_1=1,0\leq x_2\leq 1\},\\
    \Gamma_{0}&=\{0\leq x_1\leq 0,x_2=0\}\cup\{0\leq x_1\leq 0,x_2=1\}.
  \end{aligned}
\end{equation*}
\par
The flow of a viscous, compressible fluid were governed by the compressible 
Navier-Stokes system, which has been extensively studied. The general existence 
results for weak solution to stationary compressible Navier-Stokes system are given 
by the work of Lions\cite{1998L} under the 
assumption that $\gamma>1$ in two dimensions and $\gamma\geq\frac{5}{3}$ in three 
dimensions. Later, this results was improved with the weaker assumptions on the 
specific heat ration $\gamma>\frac{3}{2}$ in three dimensions by 
Feireisl\cite{2001F} for nonsteady equations. By adapting Feireisl's nonsteady 
approach, the authors in \cite{NN2002} show that $\gamma>\frac{3}{2}$ is 
also valid  to the case of steady system. An overview of these results is 
given in the monograph \cite{NS2004,2004F}. Several refined results for steady 
compressible flow were given in \cite{JMW2010,JZ2011,PW2015}. Let us remark that all 
those results were obtained under various homogeneous boundary condition.
\par
The strong solutions of steady compressible Navier-Stokes equations with homogeneous 
boundary condition has been studied in 
\cite{2000CJ,2011JK,1996NNK,1986VZ,1987V}. As to the existence of strong 
solutions to the stationary Navier-Stokes system  with the inhomogeneous boundary 
conditions, the authors in \cite{KK1997} proved the existence of strong solutions to 
the stationary problems with an inflow boundary conditions for the density and 
Dirichlet boundary condition for the velocity in a smooth two dimension domain 
$\Omega$ under the assumption that the Reynolds number is small. A mass of 
researches showed that the regularity of the strong solutions is restricted by the 
geometry of the boundary\cite{KS2006}. In \cite{KK2002,2004K,2018OJ}, the authors 
studied the existence and regularity of solutions to an inflow boundary value 
problem under the assumption that the viscosity coefficient $\mu$ is large enough on 
a polygon domain. In \cite{P2009}, T. Piasecki proved the existence of strong 
solution around a constant equilibrium with an inflow boundary condition for the 
density and 
the full Navier boundary conditions for  the velocity filed under the assumption 
that the friction coefficient $\alpha$ is large enough. Later, the authors obtained 
similar results in a cylinder domains \cite{P2010}. It is worth mentioning that the 
validity for compressible perturbation of a Poiseuille  type flow under the same 
boundary conditions as before  is also obtained in \cite{MP2014}. In contrast to the 
boundary conditions prescribed above, the authors  established the existence of 
strong solutions near the constant state ${\bf u}=0, \rho=1$ with Dirichlet boundary 
condition on $\Gamma_{in}$ and $\Gamma_{out}$ while  slip without friction boundary 
conditions on the wall $\Gamma_0$  in \cite{GJZ2014}. Recently, the authors in 
\cite{2018FN} studied the 
existence of weak solutions to the stationary compressible Navier–Stokes system for 
arbitrarily 
large boundary data under additional physical hypotheses called molecular hypothesis 
and positive compressibility in 2D or 3D domains.  One can refer to 
\cite{2015DJJY,PRS2008} for more results concerning the existence of strong 
solutions with inhomogeneous boundary conditions.
\par
The goal of this paper is to investigate the existence of strong solutions near the 
shear flow to the steady isentropic compressible Navier-Stokes system with inflow 
boundary conditions in a square. We impose the Dirichlet boundary condition on 
inflow and outflow part of the boundary, and the full Navier  boundary conditions on 
the wall $\Gamma_{0}$ for the velocity.  Due to the hyperbolicity of the density in 
the continuity equation, it is natural to prescribe the density on the inflow part 
of the boundary. Let us remark here that  we do not need any restrictions on the 
amplitude of friction coefficients $\alpha$, and an assumption that $\mu>\frac{1}{2 
\pi^2}$ is enough. It is worth noting that most of the results discussed above 
investigated the structural stability around a constant  equilibrium, while our 
background solution is a shear flow.
\par Let us introduce the perturbed flow. Consider the shear flow $(\rho_0=1,U_0)$, with $U_0=(1+x_2,0)=(\bar U ,0)$. It is obviously that the shear flow $(\rho, U_0)$ satisfy the following system:
\begin{equation*}
\begin{aligned}
\operatorname{div}(\rho_0 {U_0}) &=0 & &\text { in } \Omega, \\
\operatorname{div}(\rho_0 {U_0} \otimes {U_0})-\mu \Delta {U_0}-(\mu+\nu) \nabla \operatorname{div} U_0 +\nabla P_0 &=0 & &\text { in } \Omega, \\
{U_0} &={U_0} & &\text { on } \Gamma_{\text {in }} \cup \Gamma_{\text {out }}, \\
\rho_0 &=1 & &\text { on } \Gamma_{\text {in }}, \\
{U_0} \cdot \mathbf{n} &=0 & &\text { on } \Gamma_{0}, \\
2\mu\mathbf{n}\cdot D({ U_0})\cdot \tau +\alpha U_0\cdot \tau &= \tilde b & &\text { on } \Gamma_{0},
\end{aligned}
\end{equation*}
where $\tilde b$ is determined when $U_0$ and $\alpha$ are given. Our aim here is to study the existence and uniqueness of strong solution in a square around the shear flow $(\rho_0, U_0)$.\,
To formulate our main results, it's convenient to define the quantity $D_0$ as
\begin{equation}\label{diff-of-inti-dad}
D_0:=\|{\bf u}_0 -U_0\|_{{2-\frac{1}{p},p};\Gamma_{in}\cap\Gamma_{out}}+\|b-\tilde b\|_{1-\frac{1}{p},p;\Gamma_{0}}+\|\rho_{in}-1\|_{{1,p};\Gamma_{in}}
\end{equation}
which measure the distance of the initial data away from the shear flow $U_{0}$.
\begin{theorem}
Suppose that $2<p<\infty$,\,if $D_0$ given by (\ref{diff-of-inti-dad}) is small enough and the viscosity coefficients $\mu>\frac{1}{2\pi^2}$, then there exists a unique solution $({\bf u},\rho)\in W^{2,p}(\Omega)\times W^{1,p}(\Omega)$ to the system $(\ref{S-NS})$ such that
\begin{equation}
\|{\bf u}-U_0\|_{{2,p};\Omega}+\|\rho-1\|_{{1,p};\Omega}\leq E(D_0),
\end{equation}
where $E(D_0)$ can be arbitrarily small provided that the quantity $D_0$ is small enough.
\end{theorem}
This paper will be organized as follows. In Section 2, we linearized the original problem (\ref{S-NS}) around the shear flow $(\rho_0,U_0)$. Next, we introduce an elegant transformation to overcome the difficulties caused by the hyperbolicity of the continuity equation. In Section 3, we establish the
basic and higher order energy estimates to the linearized equations. In order to 
overcome the difficulties caused by  the nonlinearity for the
Navier–Stokes system, we derive the uniform
$W^{2,p}(\Omega) \times W^{1,p}(\Omega),(2 < p < \infty)$ estimate for the solution 
of the linearized system.  An iteration scheme is also developed to study the 
existence and uniqueness of strong solution. In Section 4, we construct approximated 
solutions by a Galerkin method and obtain the existence and regularity of weak 
solution to the linearized system. In the Section 5, we show that the iteration 
scheme designed in the Section 3 is a contraction and prove our main result Theorem 
1.1 finally.
\section{Reformulation of problem}
It is convenient to convert the inhomogeneous boundary conditions of the original problems into homogeneous one. To this end, introduce a function $\tilde{\bf u}\in W^{2,p}(\Omega)$ such that
\begin{equation}\label{homo-BC}
\begin{aligned}
&\tilde{\mathbf{u}}=\mathbf{u}_0 - U_0 &   &on\,\Gamma_{in}\cup \Gamma_{out},\\
&\tilde{\mathbf{u}} \cdot \mathbf{n} =0  &  &on\, \Gamma_0,\\
&2\mu\mathbf{n}\cdot D(\tilde{\bf u})\cdot \tau +\alpha\tilde{\bf u}\cdot \tau = 0 & &\text { on } \Gamma_{0}.
\end{aligned}
\end{equation}
Indeed, one can construct $\tilde{\bf u}$ as a solution of the following Lam$\acute{e}$ system
\begin{equation*}
-\mu \Delta \tilde{\bf u}-(\mu+\nu) \nabla \operatorname{div}\tilde{\bf u}=0 \quad\quad\text { in } \Omega, \\
\end{equation*}
with boundary condition (\ref{homo-BC}).Consequently,
\begin{equation}
\|\tilde{\bf u}\|_{W^{2,p}(\Omega)}\leq C|{\bf u}_0 -U_0|_{W^{2-\frac{1}{p},p}(\Gamma_{in}\cap{\Gamma_{out}})}.
\end{equation}
We denotes the differences
$${\mathbf{v}}=\mathbf{u}-\tilde{\mathbf{u}}-U_0, \, w=\rho-1. $$
Then, one can yields $(\bar{\mathbf{u}}, \bar{\rho})$ satisfy the following system 
by straightforward computation,
\begin{equation}\label{Lin-Homo-sys}
\begin{aligned}
\ \bar U\partial_1 w+({\bf v}+\tilde{\mathbf{u}})\cdot \nabla w + \operatorname{div}{\bf v} &=F({\bf v},w) & &\text { in } \Omega, \\
\ \bar U \partial_1 {\bf v}+{\bf{v}}\cdot \nabla U_0+\gamma \nabla w-\mu \Delta {\bf v}-(\mu+\nu) \nabla \operatorname{div}{\bf v} &=G({\bf v},w) & &\text { in } \Omega, \\
\ {\bf v} &=0 & &\text { on } \Gamma_{\text {in }} \cup \Gamma_{\text {out }}, \\
\ w &=\rho_{\text {in }}-1 & &\text { on } \Gamma_{\text {in }}, \\
\ {\bf v} \cdot \mathbf{n} &=0 & &\text { on } \Gamma_{0},\\
2\mu\mathbf{n}\cdot D({\bf v})\cdot \tau +\alpha{\bf v}\cdot \tau &= B & &\text { on } \Gamma_{0},
\end{aligned}
\end{equation}
where
\begin{equation}\label{Souc}
\begin{aligned}
B=&b-\tilde b,\\
F({\bf v}, w)=&-\operatorname{div} \tilde{\mathbf{u}}-w \operatorname{div}({\bf v}+\tilde{\mathbf{u}}), \\
G({\bf v}, w)=&-(w+1)\left({\bf v}+\tilde{\mathbf{u}}+U_{0}\right) \cdot \nabla \tilde{\mathbf{u}}-(w+1)({\bf v}+\tilde{\mathbf{u}}) \cdot \nabla {\bf v}\\
&-w U_{0} \cdot \nabla {\bf v}-\tilde{\bf{u}}\cdot \nabla U_0 -w({\bf v}+\tilde{\mathbf{u}})\cdot \nabla U_0\\
&+\mu \Delta \tilde{\mathbf{u}}+(\mu+\nu) \nabla \operatorname{div} \tilde{\mathbf{u}}+\gamma\left[(w+1)^{\gamma-1}-1\right] \nabla w.
\end{aligned}
\end{equation}
A straightforward computation gives the estimates of $F$ and $G$ as follows.
\begin{lemma}\label{FG-est}
Let $F({\bf v},w)$ and $G({\bf v},w)$ be  defined by (\ref{Souc}).Then we have
\begin{equation}\label{est-sou-x}
\begin{aligned}
&\| F({\bf v},w)\|_{W^{1,p}}+\| G({\bf v},w)\|_{L^p}\\
&\leq C[(\|{\bf  v}\|_{W^{2,p}}+\|w\|_{W^{1,p}})^3+(\|{\bf  v}\|_{W^{2,p}}+\|w\|_{W^{1,p}})^2]\\
&+E(D_0)(\|{\bf  v}\|_{W^{2,p}}+\|w\|_{W^{1,p}})+\|\tilde{\bf u}\|_{W^{2,p}}.
\end{aligned}
\end{equation}
\end{lemma}
\begin{proof}
The estimates is almost obviously by the embedding $W^{1,p}(\Omega)\hookrightarrow L^{\infty}(\Omega)$, one can easily check it term by term. In fact, the term $w\divg {\bf v}$ in $F({\bf v},w)$ is estimated as
  \begin{equation}
  \begin{aligned}
  \|w\divg {\bf v}\|_{1;p}&\leq \|w\|_{\infty}\|\nabla{\bf v}\|_{p}+\|\nabla w\|_{p}\|\nabla {\bf v}\|_{\infty}+\|w\|_{\infty}\|\nabla^2{\bf v}\|_{p}\\
  &\leq C(\|{\bf  v}\|_{W^{2,p}}+\|w\|_{W^{1,p}})^2.
    \end{aligned}
  \end{equation}
  The most subtle term in $G({\bf v},w)$ is controlled by
  \begin{equation}
  \begin{aligned}
    \|w{\bf v}\cdot\nabla{\bf v}\|_{p}&\leq\|w\|_{\infty}\|{\bf v}\|_{p}\|\nabla{\bf v}\|_{\infty}\\
    &\leq C\|{\bf  v}\|_{2;p}^2\|w\|_{1;p}\\
    &\leq C(\|{\bf  v}\|_{W^{2,p}}+\|w\|_{W^{1,p}})^3,
    \end{aligned}
  \end{equation}
  hence, the results follows.
\end{proof}
We want to prove the existence of solution by fixed point theorem. However, a  direct application of fixed point arguments may fail due to the term ${\bf v}\cdot \nabla w$ on the left hand side of (\ref{Lin-Homo-sys}) which would cause the loss of regularity on the right hand side of the system. Fortunately, we can introduce a good enough transformation such that the problematic term disappear in the new coordinates. Indeed, consider the following transformation $\psi:(z_1,z_2)\mapsto(x_1,x_2)$ satisfy
\begin{equation}\label{Chge-Varb}
\begin{cases}
x_1=\psi^{1} (z_1,z_2)=z_1,\\
x_2=\psi^{2}(z_1,z_2)=z_2+\int_{0}^{z_1}\tilde v(t,\psi^2(t;z_2))dt,\\
\end{cases}
\end{equation}
where $\tilde v = \frac{({\bf v +\tilde{\bf u}})^{(2)}}{\bar U + ({\bf v +\tilde {\bf u}})^{(1)}}$. Thanks to the smallness assumption, the transformation $\psi:\tilde\Omega\rightarrow\Omega$ is a diffeomorphism. Furthermore, if we denote
\begin{equation}\label{jacb-of-trnsf}
{\bf E}=
\left(
  \begin{array}{ccc}
    0 & 0 \\
    \tilde v(\psi(z)) &\partial_{z_2}\int_{0}^{z_1} \tilde v(\psi(t,z_2))dt \\
  \end{array}
\right),
\end{equation}
with $\nabla\psi={\bf Id}+{\bf E}$, we have the following lemma.
\begin{lemma}
Suppose that $\psi$ is defined in (\ref{Chge-Varb}),\,then $\psi$ is a diffeomorphism such that $\psi(\tilde\Omega)=\Omega$, moreover
\begin{equation}\label{small-jaco}
\|{\bf E}\|_{W^{1,p}(\tilde\Omega)}\leq E,
\end{equation}
here $E=E(\|{\bf v}\|_{2,p},\|\tilde{\bf u}\|_{2,p})$ can be arbitrary small provided $\|{\bf v}\|_{2,p},\|\tilde{\bf u}\|_{2,p}$ small enough.
\end{lemma}
\begin{proof}
Since $\bar U$ is strictly positive and $\|{\bf v}\|_{2,p},\|\tilde{\bf u}\|_{2,p}$ is small, there exists a positive constant $c$, such that
\begin{equation*}
  \bar U +({\bf v}+\tilde{\bf u})^{(1)}\geq c>0,
\end{equation*}
it follows that $\tilde v\in W^{2,p}$ and
\begin{equation}\label{small-tilde}
  \|\tilde v\|_{2,p}\leq E(\|{\bf v}\|_{2,p},\|\tilde{\bf u}\|_{2,p}).
\end{equation}
At this stage,the first part is standard, one can refer to \cite{P2010}, and the estimate (\ref{small-jaco}) for $E_{21}$ is immediately. On the other hand,
\begin{equation}
\begin{aligned}
  |E_{22}|^p&\leq|\int_{0}^{z_1}\nabla\tilde v(\psi(t,z_2))\cdot\partial_{z_2}\psi(t,z_2)dt|^p\\
  &\leq |z_1|^{p-1}\int_{0}^{z_1}|\nabla\tilde v(\psi(t,z_2))\cdot\partial_{z_2}\psi(t,z_2)|^pdt\\
  &\leq C\|\nabla\tilde v\|_{\infty}^p\int_{0}^{z_1}|1+E_{22}|^pdt,
  \end{aligned}
\end{equation}
Integrating it over $\Omega$ gives
\begin{equation}
  \|E_{22}\|_{p}\leq C\|\nabla\tilde v\|_{\infty}\|1+E_{22}\|_{p}.
\end{equation}
Differentiate $E_{22}$ with respect to $z_1$ we have
\begin{equation*}
  \begin{aligned}
    \partial_{z_1}E_{22}=\nabla\tilde v(\psi(z_1,z_2))\cdot\partial_{z_2}\psi(z_1,z_2),
  \end{aligned}
\end{equation*}
which gives
\begin{equation}
  \|\partial_{z_1}E_{22}\|_{p}\leq C\|\nabla\tilde v\|_{\infty}\|1+E_{22}\|_{p}\leq E,
\end{equation}
similarly, differentiate $E_{22}$ with respect to $z_2$ we have
\begin{equation*}
  \begin{aligned}
  \partial_{z_2}E_{22}&=\partial_{z_2}\int_{0}^{z_1} \partial_{x_i}\tilde v(\psi(t,z_2))\partial_{z_2}\psi^i(t,z_2)dt\\
&=\int_{0}^{z_1} \partial_{x_i,x_j}\tilde v(\psi(t,z_2))\partial_{z_2}\psi^j(t,z_2)\partial_{z_2}\psi^i(t,z_2)dt\\
&+\int_{0}^{z_1} \partial_{x_i}\tilde v(\psi(t,z_2))\partial^2_{z_2}\psi^i(t,z_2)dt\\
&:=I_1+I_2,
  \end{aligned}
\end{equation*}
it follows that
\begin{equation*}
  \begin{aligned}
    |I_1|^p&=|\int_{0}^{z_1} \partial_{x_i,x_j}\tilde v(\psi(t,z_2))\partial_{z_2}\psi^j(t,z_2)\partial_{z_2}\psi^i(t,z_2)dt|^p\\
    &\leq \|\partial_{z_2}\psi\|^{2p}_{\infty}|z_1|^{p-1}\int_{0}^{z_1}|\nabla_x^2\tilde v|^pdt\\
    &\leq C\int_{0}^{z_1}|\nabla_x^2\tilde v|^pdt
  \end{aligned}
\end{equation*}
and
\begin{equation*}
  \begin{aligned}
    |I_2|^p&=|\int_{0}^{z_1} \partial_{x_i}\tilde v(\psi(t,z_2))\partial^2_{z_2}\psi^i(t,z_2)dt|^p\\
   &\leq C\|\nabla\tilde v\|_{\infty}\int_{0}^{z_1}|\partial_{z_2}E_{22}|^{p}
  \end{aligned}
\end{equation*}
which implies
\begin{equation}
  \|\partial_{z_2}E_{22}\|_{p}\leq C(\|\tilde v\|_{2,p}+\|\nabla\tilde v\|_{\infty}\|\partial_{z_2}E_{22}\|_{p}),
\end{equation}
finally, combine all those inequality together the estimate (\ref{small-jaco}) for $E_{22}$ is immediate and the proof thus is completed.
\end{proof}
\begin{remark}
Since the transformation $\psi$ is a diffeomorphism, there exists a inverse transform $\phi$ such that $\psi\circ\phi={\bf Id}$, and
\begin{equation}
  \nabla\phi=\nabla\psi^{-1}={\bf Id}+\tilde{\bf E}.
\end{equation}
Denote $J=|\nabla \psi|$, then $\nabla\psi^{-1}$ can be explicitly computed as
\begin{equation}\label{jacb-of-inv-trnsf}
\nabla\psi^{-1}=\frac{1}{J}
\left(
  \begin{array}{ccc}
    1+\partial_{z_2}\int_{0}^{z_1} \tilde v(\psi(t,z_2))dt & 0 \\
    -\tilde v(\psi(z)) & 1\\
  \end{array}
\right),
\end{equation}
hence, the estimates (\ref{small-jaco}) also holds for $\tilde{\bf E}$. We also denote the transform $\psi$ as $\psi_{{\bf v}+\tilde{\bf u}}$ to emphasize that $\psi_{{\bf v}+\tilde{\bf u}}$ is induced by $\tilde v$ through (\ref{Chge-Varb}).
\end{remark}
Observe that the most important property of the transform $\psi$ is
\begin{equation}
\partial_{z_1}=\frac{\partial{x_1}}{\partial{z_1}}\partial_{x_1} +\frac{\partial{x_2}}{\partial{z_1}} \partial_{x_2}=\partial_{x_1} +\tilde v \partial_{x_2},
\end{equation}
where $\tilde v = \frac{({\bf v +\tilde{\bf u}})^{(2)}}{\bar U + ({\bf v +\tilde {\bf u}})^{(1)}}$.
Hence we can change variable and rewrite the system (\ref{Lin-Homo-sys})\,in coordinates z, which leads to
\begin{equation}\label{Z-cordin-sys}
\begin{aligned}
&((\bar U +({\bf v +\tilde{\bf u}})^{(1)})\circ\psi_{{\bf v}+\tilde{\bf u}})\partial_{z_1}w+\operatorname{div}_{z} {\bf v} =\tilde F({\bf v},w) \quad\text { in } \tilde\Omega, \\
&(\bar U\circ\psi_{{\bf v}+\tilde{\bf u}}) \partial_{z_1} {\bf v}+{\bf{v}}\cdot \nabla_{z} U_0+\gamma \nabla_{z} w-\mu \Delta_{z} {\bf v}-(\mu+\nu) \nabla \operatorname{div}_{z}{\bf v} =\tilde G({\bf v},w) \quad\text { in } \tilde\Omega, \\
&{\bf v} =0 \quad\text { on } \tilde\Gamma_{\text {in }} \cup \tilde\Gamma_{\text {out }}, \\
& w =\rho_{\text {in }}-1 \quad\text { on } \tilde\Gamma_{\text {in }}, \\
& {\bf v} \cdot \mathbf{n} =0 \quad\text { on } \tilde\Gamma_{0},\\
&2\mu\mathbf{n}\cdot D_z({\bf v})\cdot \tau +\alpha{\bf v}\cdot \tau =\tilde B \quad\text { on } \tilde\Gamma_{0}.
\end{aligned}
\end{equation}
Here $\tilde B=B-2\mu {\bf n}\cdot R({\bf v},D)\cdot \tau $,\, and the function 
$R(\cdot,\cdot )$ with the first variable denotes a function and the second is a 
differential operator representing the differences of  the differential operator 
acting  on the function in x-coordinates and z-coordinates. For instance,\,$R({\bf 
v},D)=D_x {\bf v} -D_z {\bf v}$, and we also have
\begin{equation}\label{souc-tilde}
\begin{aligned}
\tilde F({\bf v}, w)=&F({\bf v},w)-R({\bf v},\operatorname{div})\\
\tilde G({\bf v}, w)=&G({\bf v},w)-\bar U R({\bf v},\partial_1)-{\bf v}\cdot R(U_0,\nabla)\\
-&\gamma R(w,\nabla)+\mu R({\bf v},\Delta)+(\mu+\nu) R({\bf v},\nabla\operatorname{div}).
\end{aligned}
\end{equation}

\begin{lemma}
Let $\tilde F$ and $\tilde G$ are given by (\ref{souc-tilde}).Then we have
\begin{equation}\label{est-soc-z}
\begin{aligned}
&\|\tilde F({\bf v},w)\|_{W^{1,p}}+\|\tilde G({\bf v},w)\|_{L^p}\\
&\leq C[(\|{\bf  v}\|_{W^{2,p}}+\|w\|_{W^{1,p}})^3+(\|{\bf  v}\|_{W^{2,p}}+\|w\|_{W^{1,p}})^2]\\
&+E(D_0)(\|{\bf  v}\|_{W^{2,p}}+\|w\|_{W^{1,p}})+\|\tilde{\bf u}\|_{W^{2,p}}
\end{aligned}
\end{equation}
\end{lemma}
\begin{proof}
A direct computation shows that under the change of variable $\psi:(z_1,z_2)\mapsto(x_1,x_2)$,\, there holds
\begin{equation}\label{deriv-in-x-z}
\begin{aligned}
\partial_{z_{i}}v^{l}&=\partial_{x_{i}}v^{l} + E_{ik}\partial_{x_{k}}v^{l},\\
\partial^{2}_{z_{i}z_{j}}v^{l}&=\partial^{2}_{x_{i}x_{j}}v^{l}+E_{ik}\partial^{2}_{x_{j}x_{k}}v^{l} +\partial_{x_{j}}E_{ik}\partial_{x_{k}}v^{l}\\
&+E_{jk}\partial^{2}_{x_{i}x_{k}}v^{l}+E_{jm}E_{ik}\partial^{2}_{x_{k}x_{m}}v^{l}+E_{jm}\partial_{x_{m}}E_{ik}\partial_{x_{k}}v^{l}
\end{aligned}
\end{equation}
where $1\leq i,j,l\leq 2$ are integers, and the repeated subscribe always sum over it's index set. The above estimates combine with (\ref{est-sou-x}) and Lemma \ref{FG-est} the results is immediate.
\end{proof}
\section{ a priori estimates}
In this section, we construct a sequence that will converge to a solution of the nonlinear system (\ref{Z-cordin-sys}).The solution sequence is defined as follows:
\begin{equation}\label{Itrn-sys}
\begin{aligned}
 &((\bar U +({{\bf v}^{n} +\tilde{\bf u}})^{(1)})\circ\psi_{{\bf v}^n+\tilde{\bf u}})\partial_{z_1}w^{n+1}+\operatorname{div}_{z} {\bf v}^{n+1} =\tilde F({\bf v}^{n},w^n) \quad\text { in } \tilde\Omega,\\
&(\bar U\circ\psi_{{\bf v}^n+\tilde{\bf u}}) \partial_{z_1} {\bf v}^{n+1}+{\bf{v}}^{n+1}\cdot \nabla_{z} U_0+\gamma \nabla_{z} w^{n+1}\\
&-\mu \Delta_{z} {\bf v}^{n+1}-(\mu+\nu) \nabla \operatorname{div}_{z}{\bf v}^{n+1} =\tilde G({\bf v}^n,w^n)  \quad\text { in } \tilde\Omega,\\
& {\bf v}^{n+1} =0\quad\text { on } \tilde\Gamma_{\text {in }} \cup \tilde\Gamma_{\text {out }},\\
& w^{n+1} =\rho_{in}-1 \quad\text { on } \tilde\Gamma_{in}\\
& {\bf v}^{n+1} \cdot \mathbf{n} =0\quad\text { on } \tilde\Gamma_{0},\\
&2\mu\mathbf{n}\cdot D_z({\bf v}^{n+1})\cdot \tau +\alpha{\bf v}^{n+1}\cdot \tau =\tilde B({\bf v}^n)\quad\text { on } \tilde\Gamma_{0},
\end{aligned}
\end{equation}
where $\tilde B({\bf v}^n)=B-2\mu {\bf n}\cdot R({\bf v}^n,D)\cdot \tau $.
\\ To show the existence of the solutions to (\ref{Itrn-sys}), we firstly deal with the following linear system:
\begin{equation}\label{Lin-sys}
\begin{aligned}
&((\bar U +\bar{\bf v}^{(1)})\circ\psi_{\bar{\bf v}})\partial_{z_1} w+\operatorname{div} {\bf v} =f \quad\text { in } \tilde\Omega, \\
&(\bar U\circ\psi_{\bar{\bf v}}) \partial_{z_1} {\bf v}+{\bf{v}}\cdot \nabla U_0+\gamma \nabla w-\mu \Delta {\bf v}-(\mu+\nu) \nabla \operatorname{div}{\bf v} = g  \quad\text { in } \tilde\Omega,\\
& {\bf v} =0  \quad\text { on } \tilde\Gamma_{\text {in }} \cup \tilde\Gamma_{\text {out }},\\
& w =w_{in} \quad\text { on } \tilde\Gamma_{\text {in }},\\
& {\bf v} \cdot \mathbf{n} =0 \quad\text { on } \tilde\Gamma_{\text {0 }},\\
&2\mu\mathbf{n}\cdot D_z({\bf v})\cdot \tau +\alpha{\bf v}\cdot \tau =\tilde B \quad\text { on } \tilde\Gamma_{0},
\end{aligned}
\end{equation}
where $f,g\in L^2(\Omega)$,\,$w_{in}\in L^2(\Gamma_{in}),\tilde B\in L^2(\Gamma_0)$ are given functions. Without the risk of confusing, we
readily remove the superscript of the domains.
\subsection{Solution of mass equation}
\ \par One of the most important advantage of changing variable is that we can define a  operator $S:v\mapsto w$ as the solution of equation
\begin{equation}\label{tran-eq}
\begin{cases}
((\bar U +\bar{\bf v}^{(1)})\circ\psi_{\bar{\bf v}})\partial_{z_1}w=v \quad\quad \text{in}\,\Omega,\\
w= w_{in}(z_2)\quad\quad on\,\Gamma_{in}.
\end{cases}
\end{equation}
Indeed,for a continuous function $v$,\, setting
\begin{equation}\label{def-of-S}
S(v)(z):= w_{in}(z_2)+\int_{0}^{z_1}\frac{v}{(\bar U +\bar{\bf v}^{(1)})\circ\psi_{\bar{\bf v}}}(t,z_2)dt,
\end{equation}
then it is easy to verfy that $S(v)$ satisfies (\ref{tran-eq}). Moreover,we have the following estimates.
\begin{lemma}
Let $S$ be defined in (\ref{def-of-S}), $w_{in}\in L^2(\Gamma_{in})$ is given then
\begin{equation}\label{estm-of-dsty}
\|S(v)\|_{L^{\infty}(L^2)(\Omega)}\leq C(| w_{in}|_{L^2(\Gamma_{in})}+\|v\|_{L^2(\Omega)})
\end{equation}
\end{lemma}
\begin{proof}
Given $z_1$, let $\Omega_{z_1}$  denote the $z_1$-cut of $\Omega$, then
\begin{equation}
\begin{aligned}
  |S(v)|^2_{L^2(\Omega_{z_1})}&=\int_{0}^{1}[w_{in}(z_2)+\int_{0}^{z_1}\frac{v}{(\bar U +\bar{\bf v}^{(1)})\circ\psi_{\bar{\bf v}}}(t,z_2)dt]^2dz_2\\
  &\leq C[| w_{in}|^2_{L^2(\Gamma_{in})}+\|v\|_{L^2(\Omega)}],
  \end{aligned}
\end{equation}
where we have used
\begin{equation*}
  \bar U +({\bf v}+\tilde{\bf u})^{(1)}\geq c>0,
\end{equation*}
for some positive constant $c>0$, and the estimates (\ref{estm-of-dsty}) follows.
\end{proof}
\begin{remark}
By density arguments one can easily extend $S$ to $L^2(\Omega)$ which also preserves 
the same estimates as (\ref{estm-of-dsty}).
\end{remark}
\subsection{$H^1$ estimates for linear system}
\begin{lemma}\label{H1-est}
Let $\bar{\bf v}\in W^{2,p}(\Omega)$ and $\| \bar{\bf v}\|_{2,p}$ be small enough, the viscous coefficients $\mu>\frac{1}{2\pi^2}$. Suppose that $({\bf v},w)$ be a solution to the system (\ref{Lin-sys}) with given $(f,g,\tilde B,w_{in})\in L^2(\Omega)\times V^{*}\times L^2(\Gamma_0)\times L^2({\Gamma_{in}})$, then
\begin{equation}\label{L2-estm}
\|{\bf v}\|_{H^1} + \|w\|_{L^{\infty}(L^2)}\leq C(\|f\|_{L^2}+\|g\|_{V^{*}}+\|\tilde B\|_{L^2(\Gamma_{0})}+\|w_{in}\|_{L^2({\Gamma_{in}})})
\end{equation}
where
$$V:=\{{\bf v}\in W^{1,2}(\Omega):{\bf 
v}=0\,\text{on}\,\Gamma_{in}\cup\Gamma_{out},\,{\bf v}\cdot n|_{\Gamma_0}=0\}$$
and $V^{*}$ is the dual space of V.
\end{lemma}
\begin{proof}
For any $\phi\in V$ we have
\begin{equation*}
\begin{aligned}
\int_{\Omega}(-\mu\Delta{\bf v}-(\mu+\nu)\nabla\operatorname{div}{\bf v})\cdot \phi dz&=\int_{\Omega}(2\mu D({\bf v}):\nabla \phi+\nu\operatorname{div}{\bf v}\operatorname{div}\phi) dz\\
&-\int_{\partial\Omega}2\mu {\bf n}\cdot D({\bf v})\cdot \phi ds.
\end{aligned}
\end{equation*}
Then multiplying (\ref{Lin-sys})\, by ${\bf v}$ and integrating over $\Omega$ 
we have 
\begin{equation}\label{L2-1}
\begin{aligned}
-&\frac{1}{2}\int_{\Omega} \partial_{z_1}(\bar U\circ\psi_{\bar{\bf v}})|{\bf v}|^2 dz+\int_{\Omega} ({\bf v}\cdot\nabla U_0)\cdot {\bf v}dz-\gamma\int_{\Omega} w\operatorname{div}{\bf v}dz\\
+&2\mu\int_{\Omega} |D({\bf v})|^2dz+\nu\int_{\Omega}|\operatorname{div}{\bf v}|^2dz+\int_{\Gamma_{0}} \alpha|{\bf v}|^2 ds=\int_{\Gamma_{0}}\tilde B{\bf v}\cdot\tau+\int_{\Omega} g\cdot {\bf v}.
\end{aligned}
\end{equation}
For the first term, we have
\begin{equation}
|\partial_{z_1}(\bar U\circ\psi_{\bar{\bf v}})|=|\partial_{x_1}\bar U + E_{1k}\partial_{x_k} \bar U|=|E_{12}|\leq E(\|{\bar{\bf v}}\|_{W^{2,p}}).
\end{equation}
As to the second term in the left hand side of (\ref{L2-1}), by applying the Poincar$\acute{e}$'s inequality we get
$$|\int_{\Omega} ({\bf v}\cdot\nabla U_0)\cdot {\bf v}dz|=|\int_{\Omega} v^{(1)} v^{(2)}dz\leq\frac{1}{2}\|{\bf v}\|^2_{L^2}\leq\frac{1}{2}\|\nabla{\bf v}\|^2_{L^2}.$$
The well known Korn inequality in \cite{MP2014,1994H} shows that
\begin{equation}\label{L2-2}
\mu\pi^2\|\nabla{\bf v}\|^2_{L^2}\leq 2\mu\int_{\Omega} |D({\bf v})|^2dz+\nu\int_{\Omega}|\operatorname{div}{\bf v}|^2dz,
\end{equation}
here the appearance of constant $\pi$ concerning the best constant in Korn inequality.
Involving in the continuity equation in (\ref{Lin-sys}) we have
\begin{equation}\label{L2-3}
\begin{aligned}
-\int_{\Omega} w\operatorname{div}{\bf v}&dz=\int_{\Omega} w((\bar U +\bar{\bf v}^{(1)})\circ\psi_{\bar{\bf v}})\partial_{z_1}wdz-\int_{\Omega} wfdz\\
=&-\frac{1}{2}\int_{\Omega} \partial_{z_1}((\bar U +\bar{\bf v}^{(1)})\circ\psi_{\bar{\bf v}})w^2dz-\frac{1}{2}\int_{\Gamma_{in}} ((\bar U +\bar{\bf v}^{(1)})\circ\psi_{\bar{\bf v}})w_{in}^2ds\\
&+\frac{1}{2}\int_{\Gamma_{out}}((\bar U +\bar{\bf v}^{(1)})\circ\psi_{\bar{\bf v}})w^2ds-\int_{\Omega} wfdz
\end{aligned}
\end{equation}
Due to the smallness assumption of ${\bar{\bf v}}$,\, the integral over 
$\Gamma_{out}$ will be nonnegative, and we also have
\begin{equation}
 |\partial_{z_1}(\bar{\bf v}^{(1)}\circ\psi_{\bar{\bf v}})|=|\partial_{x_1}\bar{\bf v}^{(1)} + E_{1k}\partial_{x_k} \bar{\bf v}^{(1)}|\leq E(\|{\bar{\bf v}}\|_{W^{2,p}}).
\end{equation}
Combing all those estimates together we arrive at
\begin{equation}\label{L2-key-1}
\begin{aligned}
\ [\mu\pi^2&-(\frac{1}{2}+\frac{E(\|\bar{\bf v}\|_{W^{2,p}})}{2})] |\nabla{\bf v}\|^2_{L^2}\\
&\leq\|g\|_{V^{*}}\|{\bf v}\|_{L^2}+|\tilde B|_{L^2({\Gamma_{0}})}\|{\bf v}\|_{L^2({\Gamma_{0}})}+E\|w\|^2_{L^2}
\\&+C|w_{in}|^2_{L^2(\Gamma_{in})}+\|w\|_{L^2}\|f\|_{L^2}
\end{aligned}
\end{equation}
Observe that the left hand side of (\ref{L2-key-1}) will be positive provided $\|{\bf v}\|_{W^{2,p}}$ is small enough and the viscosity efficiencies  $\mu>\frac{1}{2\pi^2}$.\,Finally by substituting $v=f-\operatorname{div} {\bf v}$ in (\ref{estm-of-dsty}) yields
\begin{equation}\label{L2-key2}
\|w\|_{L^{\infty}(L^2)}\leq C(|w_{in}|_{L^2(\Gamma_{in})}+\|\nabla{\bf v}\|_{L^2(\Omega)}++\|f\|_{L^2(\Omega)}).
\end{equation}
Combing (\ref{L2-key-1}) and (\ref{L2-key2}) by using trace theorem and Poincar$\acute{e}$'s inequality we get the estimates (\ref{L2-estm}).
\end{proof}
Next, we want to show higher regularity up to boundary by the method of difference quotient.
\begin{lemma}\label{diff-quo}
Suppose that $({\bf v}, w)\in H^{1}\times L^{\infty}(L^2)$ be a solution to the the system (\ref{Lin-sys}) with $(f,g,\tilde B,w_{in})\in H^1(\Omega)\times L^2(\Omega)\times H^{\frac{1}{2},2}(\Gamma_0)\times H^1({\Gamma_{in}})$,\,then we have
\begin{equation}\label{H2-est}
\|w\|_{H^1(\Omega)}+\|{\bf v}\|_{H^2(\Omega)}\leq C(\|f\|_{H^1}+\|g\|_{L^2}+|\tilde B|_{H^{\frac{1}{2},2}(\Gamma_{0})}+|w_{in}|_{H^1({\Gamma_{in}})}).
\end{equation}
\end{lemma}
\begin{proof}
Let us firstly focus on the interior estimates. By using the standard $L^2$ theory of elliptic system, we only need to find a bound of $\|\partial_{z_2}w\|_{L^2}$. To this end, define the difference quotient operator along $z_2$ direction $\Delta_{2}^{-h}$ as
$$\Delta_{2}^{-h}\phi:=-\frac{\phi(z_1,z_2 -h)-\phi(z_1,z_2)}{h}.$$
Effecting $\Delta_{2}^{h}$ on the mass equation in (\ref{Lin-sys}), we have
\begin{equation}\label{def-quo-2}
(\bar U +\bar{\bf 
v}^{(1)})\partial_{z_1}(\Delta_{2}^{h}w)+\partial_{z_1}w^h\Delta_{2}^{h}(\bar U 
+\bar{\bf 
v}^{(1)})+\operatorname{div} (\Delta_{2}^{h}{\bf v}) 
=\Delta_{2}^{h}f,
\end{equation}
where $w^h(z_1,z_2)=w(z_1,z_2+h)$. Hence, in order to find a bound of  
$\|\partial_{z_2}w\|_{L^2}$, it is sufficient to find a bound of 
$\|\partial_{z_2}{\bf v}\|^2_{H^1}$. Since $({\bf v}, w)\in H^{1}\times 
L^{\infty}(L^2)$ be a solution of  the system (\ref{Lin-sys}), which means
\begin{equation}\label{def-wk-solu-1}
\begin{aligned}
\int_{\Omega} \bar U \partial_{z_1}{\bf v}& \cdot \phi dz +\int_{\Omega}  ({\bf v}\cdot\nabla U_0)\cdot \phi dz-\gamma\int_{\Omega} w\operatorname{div} \phi dz+\mu\int_{\Omega} \nabla{\bf v} :\nabla \phi dz\\
&+(\mu+\nu)\int_{\Omega} \operatorname{div}{\bf v}\operatorname{div}\phi dz=\int_{\Omega}  g\cdot \phi dz
\end{aligned}
\end{equation}
holds for any $\phi\in C_c^{\infty}(\Omega)$. There is no risk of confusing, we may remove the superposition of $\psi$.
Replace $\phi$ in (\ref{def-wk-solu-1}) by $\Delta_{2}^{-h}\phi$,\,then we have
\begin{equation}\label{dii-quti-1}
\begin{aligned}
-&\gamma\int_{\Omega} w\operatorname{div} (\Delta_{2}^{-h}\phi) dz+\mu\int_{\Omega} \nabla{\bf v} :\nabla (\Delta_{2}^{-h}\phi) dz+(\mu+\nu)\int_{\Omega} \operatorname{div}{\bf v}\operatorname{div}(\Delta_{2}^{-h}\phi) dz\\
=&\gamma\int_{\Omega} (\Delta_{2}^{h}w)\operatorname{div} \phi dz-\mu\int_{\Omega} \nabla(\Delta_{2}^{h}{\bf v}) :\nabla \phi dz-(\mu+\nu)\int_{\Omega} \operatorname{div}(\Delta_{2}^{h}{\bf v})\operatorname{div}\phi dz\\
=&-\int_{\Omega} \bar U \partial_{z_1}{\bf v} \cdot (\Delta_{2}^{-h}\phi) dz -\int_{\Omega}  ({\bf v}\cdot\nabla U_0)\cdot (\Delta_{2}^{-h}\phi) dz+\int_{\Omega} g\cdot \Delta_{2}^{-h}\phi.
\end{aligned}
\end{equation}
For any interval $I\subset\subset(0,1)$, selecting a cut-off function $\eta(t)\in C_c^{\infty}(0,1)$,\, such that $0\leq\eta\leq 1$ and $\eta\equiv 1$ in $I$.
 Substituting $\phi=\eta^2(z_2)\Delta_{2}^{h}{\bf v}$ in (\ref{dii-quti-1}) and 
 combing with (\ref{def-quo-2}), a direct computation shows that
\begin{equation}
\begin{aligned}
\gamma\int_{\Omega} (\Delta_{2}^{h}w)\operatorname{div} \phi dz&=\gamma\int_{\Omega} (\Delta_{2}^{h}w)[2\eta(\Delta_{2}^{h}{\bf v})\cdot\nabla\eta+\eta^2 \operatorname{div}\Delta_{2}^{h}{\bf v}]  dz\\
&=\gamma\int_{\Omega} (\Delta_{2}^{h}w)[2\eta(\Delta_{2}^{h}{\bf 
v})\cdot\nabla\eta -\eta^2(\bar U +\bar{\bf 
v}^{(1)})\partial_{z_1}(\Delta_{2}^{h}w)]  dz\\
&+\gamma\int_{\Omega} (\Delta_{2}^{h}w)[\eta^2\Delta_{2}^{h}f 
-\eta^2\partial_{z_1}w^h\Delta_{2}^{h}(\bar U 
+\bar{\bf 
	v}^{(1)})]  dz\\
&=\gamma\int_{\Omega} (\Delta_{2}^{h}w)[2\eta(\Delta_{2}^{h}{\bf 
v})\cdot\nabla\eta+\eta^2\Delta_{2}^{h}f-\eta^2\partial_{z_1}w^h\Delta_{2}^{h}(\bar 
U 
+\bar{\bf 
v}^{(1)}) ]  dz\\
&+\frac{\gamma}{2}\int_{\Omega} \partial_{z_1} 
(\bar U +\bar{\bf 
v}^{(1)})|\eta\Delta_{2}^{h}w|^2-\frac{\gamma}{2}\int_{\Gamma_{out}}  (\bar U 
+\bar{\bf 
v}^{(1)})|\eta\Delta_{2}^{h}w|^2 ds,\\
&+\frac{\gamma}{2}\int_{\Gamma_{in}} 
(\bar U +\bar{\bf v}^{(1)})|\eta\Delta_{2}^{h}w|^2 ds 
\end{aligned}
\end{equation}
and
\begin{equation}\label{def-quo-4}
\begin{aligned}
-\mu&\int_{\Omega} \nabla(\Delta_{2}^{h}{\bf v}) :\nabla \phi dz-(\mu+\nu)\int_{\Omega} \operatorname{div}(\Delta_{2}^{h}{\bf v})\operatorname{div}\phi dz\\
=&-\mu\int_{\Omega} \eta^2|\nabla(\Delta_{2}^{h}{\bf v})|^2 dz-(\mu+\nu)\int_{\Omega} \eta^2|\operatorname{div}(\Delta_{2}^{h}{\bf v})|^2dz\\
+&2\mu\int_{\Omega} \eta\nabla(\Delta_{2}^{h}{\bf v}) :(\nabla\eta\Delta_{2}^{h} {\bf v})dz+2(\mu+\nu)\int_{\Omega} \eta\operatorname{div}(\Delta_{2}^{h}{\bf v})\Delta_{2}^{h} {\bf v}\cdot\nabla\eta dz,
\end{aligned}
\end{equation}
Combing (\ref{dii-quti-1})-(\ref{def-quo-4}), we obtain
\begin{equation}\label{def-quo-5}
\begin{aligned}
\int_{\Omega} \eta^2|\nabla(\Delta_{2}^{h}{\bf v})|^2&\leq C(\delta\|\eta \nabla(\Delta_{2}^{h}){\bf v}\|^2_{L^2}+\|{\bf v}\|^2_{H^1}+\|\partial_{z_2}f\|^2_{L^2}+\|g\|^2_{L^2}+|\partial_{z_2}w_{in}|^2_{L^2(\Gamma_{in})})\\
&\leq C(\|f\|^2_{H^1}+\|g\|^2_{L^2}+\|\tilde B\|^2_{L^2(\Gamma_{0})}+\|w_{in}\|^2_{H^1({\Gamma_{in}})}),
\end{aligned}
\end{equation}
here $\delta$ is a arbitrary small constant. Letting $h\rightarrow 0 $ in 
$(\ref{def-quo-5})$,\, we get
\begin{equation*}
\|\eta \partial_{z_2}{\bf v}\|^2_{{H^1}(\Omega)}\leq C(\|f\|^2_{{H^1}(\Omega)}+\|g\|^2_{{L^2}(\Omega)}+\|\tilde B\|^2_{L^2(\Gamma_{0})}+\|w_{in}\|^2_{H^1({\Gamma_{in}})}),
\end{equation*}
denote $\Omega^{\prime}=(0,1)\times I$, which implies that
\begin{equation}
\|\partial_{z_2}{\bf v}\|^2_{{H^1}(\Omega^{\prime})}\leq C(\|f\|^2_{{H^1}(\Omega)}+\|g\|^2_{{L^2}(\Omega)}+\|\tilde B\|^2_{L^2(\Gamma_{0})}+\|w_{in}\|^2_{H^1({\Gamma_{in}})}).
\end{equation}
On the other hands, (\ref{def-quo-2}) implies that
\begin{equation}
\|\partial_{z_2}w\|_{L^2(\Omega^{\prime})}\leq C(\|\partial_{z_2}{\bf v}\|_{{H^1}(\Omega^{\prime})}+\|f\|_{{H^1}(\Omega)}+|\partial_{z_2}w(0,z_2)|_{L^2(\Gamma_{in})}),
\end{equation}
which combing with the standard $H^2$ estimates of elliptic system gives the interior estimate of (\ref{H2-est}) in $\Omega^{\prime}$. \\
As to the boundary estimates near $\Gamma_{0}^{-}:=\{(z_1,z_2)|0<z_1<1,z_2=0\}$, we 
extend the domain $\Omega$ to $\Omega^*=(0,1)\times(-1,1)$, and denote  the even 
extension of $w,v_1$, the odd extension of $v_2$ with respect to $z_1=0$ as 
$w^*,v_1^*,v_2^*$ respectively. More precisely,
\begin{equation*}
\begin{aligned}
  v_1^*(z_1,z_2)&=
  \begin{cases}
    v_1(z_1,z_2)\quad\quad\, 0<z_2<1,\\
    \hat v_1(z_1,z_2)=v_1(z_1,-z_2)\quad -1<z_2<0,
  \end{cases}\\
  v_2^*(z_1,z_2)&=
  \begin{cases}
    v_2(z_1,z_2)\quad\quad\,\, 0<z_2<1,\\
    \hat v_2(z_1,z_2)=-v_2(z_1,-z_2)\quad -1<z_2<0.
  \end{cases}
  \end{aligned}
\end{equation*}
Due to the boundary condition in (\ref{Lin-sys}), we have $v_2=0$ on $\Gamma_0^-$ and the Navier boundary conditions preserve this symmetry, that is to say
\begin{equation*}
  2\mu\mathbf{n}\cdot D({\bf v})\cdot \tau +\alpha{\bf v}\cdot \tau=2\mu\hat{\mathbf{ n}}\cdot D({\hat{\bf v}})\cdot \tau +\alpha{\hat{\bf v}}\cdot \tau\quad\text{on}\,\,\Gamma_0^-,
\end{equation*}
where $\hat{\mathbf{ n}}$ is the outer normal vector of the domain $(0,1)\times(-1,0)$ on $\Gamma_0^-$, which implies that
\begin{equation*}
  \partial_{z_2}v_1^*=-\frac{\tilde B}{\mu}+\frac{\alpha}{\mu}v_1+\partial_1v_2\quad\text{on}\,\,\Gamma_0^-.
\end{equation*}
Hence the extended function $(v_1^*,v_2^*)$ still belongs to $H^{1}(\Omega^*)$, and it is easy to check that $({\bf v}^*,w^*)$ satisfy the system (\ref{Lin-sys}) in the senses of $\mathfrak{D}^{\prime}(\Omega^*)$ with corresponding extension of source terms. It follows that the estimates near the boundary $\Gamma_0^-$ is converted to an interior one, a slight modification of the above proof derives the desired estimates. An identical argument can be applied on $\Gamma_0^+:=\{(z_1,z_2)|0<z_1<1,z_2=1\}$ and we get the global regularity as claimed.
\end{proof}
\subsection{$W^{2,p}$ estimates of linear system}
The crucial point to establish the $W^{2,p}$ estimates is to find a bound of  $\|\partial_{z_2}w\|_{L^p}$, a key observation shows that it can be bounded by $\|\operatorname{curl}\partial_{z_2}{\bf v}\|_{L^p}$. To this end, we introduce the following lemma concerning the so called Bogovskii operator which is proved by Bogovskii in \cite{1979B}.
\begin{lemma}
  Suppose that $\Omega\subset\mathbb{R}^n$ is starlike with respect to some ball contained in it, and that $1<p<\infty$ and $n\geq2$, Then, the exists a constant $C>0$, depending only on $n, p$ and $\Omega$, such that for any $h\in L^p(\Omega)$ with $\int_{\Omega}h(x) dx=0$, there is a vector field $\omega\in W_{0}^{1,p}$ satisfying
\begin{equation*}
  \begin{aligned}
    \begin{cases}
      \divg \omega =h  \quad&\text{in}\quad \Omega,\\
      \omega=0  \quad&\text{on}\quad \partial\Omega.
    \end{cases}
  \end{aligned}
\end{equation*}
and
\begin{equation*}
\|\omega\|_{W^{1,p}(\Omega)}\leq C\|h\|_{L^p(\Omega)},
\end{equation*}
where the constant $C$ depends only on $p$ and $\Omega$.
\end{lemma}
By constructing a proper test function defined by Bogovskii operator, we can establish the following results.
\begin{lemma}
Assume that $({\bf v},w)$ be a solution of (\ref{Lin-sys}) with $(f,g,\tilde B,w_{in})\in  W^{1,p}(\Omega)\times L^p(\Omega)\times H^{\frac{1}{2},2}(\Gamma_0)\times W^{1,p}(\Gamma_{in})$, then for any $p\in(2,\infty)$ we have
\begin{equation}\label{w2p}
\begin{aligned}
\|&\partial_{z_2}w\|_{L^p}+|\partial_{z_2}w|_{L^p(\Gamma_{out})}\\
&\leq C(\|\operatorname{curl}\partial_{z_2}{\bf v}\|_{L^p}+\|f\|_{W^{1,p}}+\|g\|_{L^p}+|\tilde B|^2_{H^{\frac{1}{2},2}(\Gamma_{0})}+|w_{in}|_{W^{1,p}({\Gamma_{in}})})
\end{aligned}
\end{equation}
\end{lemma}
\begin{proof}
Setting $\phi=\mathfrak{B}(\partial_{z_2}w|\partial_{z_2}w|^{p-2}-M)$,\,then by the properties of Bogovskii operator we have
\begin{equation}\label{est-bgski}
\|\phi\|_{W^{1,\frac{p}{p-1}}(\Omega)}\leq C\|\partial_{z_2}w\|^{p-1}_{L^p(\Omega)},
\end{equation}
where $M=\int_{\Omega} \partial_{z_2}w|\partial_{z_2}w|^{p-2} dz$.\, Next, 
differentiate the momentum equation in (\ref{Lin-sys}) with respect to $z_2$ and 
multiply the resulting equations by  $\phi$ we have
\begin{equation}\label{w2p-1}
\begin{aligned}
-&\gamma\int_{\Omega} |\partial_{z_2}w|^pdz+\gamma M\int_{\Omega} |\partial_{z_2}w|dz+(2\mu+\nu)\int_{\Omega}\partial_{z_2}w|\partial_{z_2}w|^{p-2} \operatorname{div}\partial_{z_2}{\bf v} dz\\
+&\mu\int_{\Omega}\partial_{z_2}(\operatorname{curl}{\bf v}):\nabla\phi  dz+(\mu+\nu)M\int_{\Omega}\operatorname{div}\partial_{z_2}{\bf v}dz+\int_{\Omega}(\partial_{z_2}\bar U)\partial_{z_1} {\bf v}\cdot\phi dz\\
+&\int_{\Omega}\bar U (\partial_{z_1,z_2}{\bf v})\cdot\phi dz+\int_{\Omega}\partial_{z_2}{\bf v}\nabla U_0 \phi  dz
+\int_{\Omega}{\bf v}\cdot\nabla(\partial_{z_2}U_0)\phi dz=-\int_{\Omega}g\partial_{z_2}\phi  dz
\end{aligned}
\end{equation}
Consequently,
\begin{equation}\label{w2p-2}
\begin{aligned}
\gamma\int_{\Omega}& |\partial_{z_2}w|^pdz-(2\mu+\nu)\int_{\Omega}\partial_{z_2}w|\partial_{z_2}w|^{p-2} \operatorname{div}\partial_{z_2}{\bf v} dz\\
&\leq(\|\operatorname{curl}\partial_{z_2}{\bf v}\|_{L^p}+\|g\|_{L^p}+\|{\bf v}\|_{W^{1,p}})\|\partial_{z_2}w\|^{p-1}_{L^p}\\
&+(\|\partial_{z_2} w\|_{L^{1}}+\|\nabla^2 {\bf v}\|_{L^1})\|\partial_{z_2} w\|^{p-1}_{L^{p-1}}.
\end{aligned}
\end{equation}
To control the second term of the l.h.s of (\ref{w2p-2}), we differentiate the mass equation in (\ref{Lin-sys}) to obtain
\begin{equation}\label{w2p-3}
\partial_{z_2}(\bar U +\bar{\bf v}^{(1)})\partial_{z_1}w + (\bar U +\bar{\bf v}^{(1)})\partial_{z_1,z_2}w+\operatorname{div}\partial_{z_2}{\bf v}=\partial_{z_2}f
\end{equation}
and hence
\begin{equation}\label{w2p-4}
\begin{aligned}
-&\int_{\Omega}\partial_{z_2}w|\partial_{z_2}w|^{p-2} \operatorname{div}\partial_{z_2}{\bf v} dz\\
&=\int_{\Omega}\partial_{z_2}w|\partial_{z_2}w|^{p-2}[\partial_{z_2}(\bar U +\bar{\bf v}^{(1)})\partial_{z_1}w + (\bar U +\bar{\bf v}^{(1)})\partial_{z_1,z_2}w-\partial_{z_2}f] dz\\
&=\int_{\Omega}\partial_{z_2}(\bar U +\bar{\bf v}^{(1)})\partial_{z_2}w|\partial_{z_2}w|^{p-2}\partial_{z_1}w dz -\int_{\Omega}\partial_{z_2}w|\partial_{z_2}w|^{p-2}\partial_{z_2}f\\
&-\frac{1}{p}\int_{\Omega}\partial_{z_1}(\bar U+\bar{\bf 
v}^{(1)})|\partial_{z_2}w|^p+\int_{\Gamma_{out}}(\bar U +\bar{\bf 
v}^{(1)})|\partial_{z_2}w|^p ds-\int_{\Gamma_{in}}(\bar U +\bar{\bf 
v}^{(1)})|\partial_{z_2}w|^p ds.
\end{aligned}
\end{equation}
Recall that $\bar U(z_1,z_2)=\bar U\circ \psi_{\bar{\bf v}}$ and the definition of 
$\psi_{\bar{\bf v}}$,\, due to the smallness 
of $\|\partial_{z_1}\bar{\bf v}^{(1)}\|_{\infty}$, we have 
\begin{equation*}
	|\partial_{z_1}(\bar U+\bar{\bf 
		v}^{(1)})|\leq E,
\end{equation*}
which combine with $(\ref{w2p-2})$ and $(\ref{w2p-4})$, after using Young's 
inequality and (\ref{H2-est}) we arrive at
\begin{equation}\label{w2p-1}
\begin{aligned}
&\int_{\Omega} |\partial_{z_2}w|^pdz+\int_{\Gamma_{out}}|\partial_{z_2}w|^p ds\\
&\leq(\|\operatorname{curl}\partial_{z_2}{\bf v}\|^p_{L^p}+\delta\|\partial_{z_1}w\|^p_{L^p}+\|g\|^p_{L^p}+\|f\|^p_{W^{1,p}}+|w_{in}|^p_{W^{1,p}}(\Gamma_{in})),
\end{aligned}
\end{equation}
here $\delta$ is a small constant depending on $\|{\bar{\bf v}}\|_{2,p}$, which 
yields $(\ref{w2p})$.
\end{proof}
The $W^{2,p}$ estimates is done if  $\|\operatorname{curl}\partial_{z_2}{\bf 
v}\|_{L^p}$ is well controlled. Roughly speaking, by taking $\operatorname{curl}$ on 
the both side of momentum equations in $(\ref{Lin-sys})$, One can find  
$\operatorname{curl} {\bf v}$ satisfies a Laplace systems, and we can expect there 
is a good estimate in the interior. But unfortunately, it's  not clear about the 
behavior of $\operatorname{curl}{\bf v}$ near $\Gamma_{in}$ and $\Gamma_{out}$. 
Nevertheless, the authors have proved by a delicate construction in \cite{GJZ2014} 
the following beautiful results,
\begin{equation}\label{w2p-key}
\|\operatorname{curl}{\bf v}\|_{W^{1,p}(\Omega)}\leq C(\epsilon\|{\bf v}\|_{W^{2,p}}+\|f\|_{W^{1,p}}+\|g\|_{L^p}+\|w_{in}\|_{W^{1,p}({\Gamma_{in}})}),
\end{equation}
where $\epsilon$ is a sufficiently small constant. Since (\ref{Lin-sys}) has exactly 
the same main terms as the linear system in \cite{GJZ2014}, one 
checks easily that the estimate (\ref{w2p-key}) still holds true for the system 
(\ref{Lin-sys}) by adapting the constructions developed in \cite{GJZ2014}.  It's 
immediately to get the $W^{2,p}$ estimates by using 
$(\ref{w2p-key})$.
\begin{lemma}
 For any $p\in(2,\infty)$, assume that $({\bf v},w)$ be a solution of (\ref{Lin-sys}) with $(f,g,\tilde B,w_{in})\in  W^{1,p}(\Omega)\times L^p(\Omega)\times W^{1-\frac{1}{p},p}(\Gamma_0)\times W^{1,p}(\Gamma_{in})$, then we have
\begin{equation}\label{final-w2p}
\begin{aligned}
\|w\|_{W^{1,p}(\Omega)}&+\|\partial_{z_1}w\|_{W^{1,p}(\Omega)}+\|{\bf v}\|_{W^{2,p}(\Omega)}\\
&\leq C(\|f\|_{W^{1,p}}+\|g\|_{L^p}+\|\tilde B\|_{W^{1-\frac{1}{p},p}(\Gamma_{0})}+\|w_{in}\|_{W^{1,p}({\Gamma_{in}})})
\end{aligned}
\end{equation}
\end{lemma}
\begin{proof}
The same extension introduced in the proof of Lemma \ref{diff-quo} shows that the 
boundary estimates can  be converted  into the interior one, hence we only 
demonstrate the interior estimates for the convenience.
By standard $W^{2,p}$ estimates of elliptic system, we have
\begin{equation}
\|{\bf v}\|_{W^{2,p}(\Omega)}\leq C(\|{\bf v}\|_{W^{1,p}}+\|g\|_{L^p}+\|\tilde 
B\|_{W^{1-\frac{1}{p},p}(\Gamma_{0})}+\|\nabla w\|_{L^p})
\end{equation}
By virtue of mass equation in (\ref{Lin-sys}) we have
\begin{equation}
\|{\bf v}\|_{W^{2,p}(\Omega)}\leq C(\|f\|_{W^{1,p}}+\|g\|_{L^p}+\|\tilde B\|_{W^{1-\frac{1}{p},p}(\Gamma_{0})}+\|{\bf v}\|_{W^{1,p}}+\|\partial_{z_2} w\|_{L^p})
\end{equation}
which combing with the estimates (\ref{w2p}) and (\ref{w2p-key}) after using interpolation inequalities gives the results.
\end{proof}
\section{Solution of linear system}
In this section, the existence of weak solution to the system (\ref{Lin-sys}) is given by Galerkin method. Having the weak solution at hand, we can show easily that this solution is also strong if the data has the appropriate regularity.
\subsection{Weak solution}
\begin{definition}
We call $(w,{\bf v})$ is a weak solution of (\ref{Lin-sys}) if for any $\phi\in V$ 
there holds
\begin{equation}\label{def-weak-v}
\begin{aligned}
\int_{\Omega}[\bar U\partial_{z_1} &w +{{\bf v}\cdot\nabla U_0}]\phi dz -\gamma\int_{\Omega}w\operatorname{div}\phi dz-\int_{\Gamma_0}\alpha({\bf v}\cdot\tau)(\phi\cdot\tau) ds\\
&+\int_{\Omega}(2\mu D({\bf v}):\nabla \phi+\nu\operatorname{div}{\bf v}\operatorname{div}\phi) dz=\int_{\Omega}g\phi dz+\int_{\Gamma_0}\tilde B(\phi\cdot\tau) ds.
\end{aligned}
\end{equation}
and  for any $\eta\in \tilde C^{\infty}(\Omega)$, there holds
\begin{equation}\label{def-weak-w}
\begin{aligned}
-\int_{\Omega} \eta\partial_{z_1}(\bar U +\bar{\bf v}^{(1)})wdz&-\int_{\Omega} \partial_{z_1}\eta(\bar U +\bar{\bf v}^{(1)}) wdz\\
&=\int_{\Omega} [f-\operatorname{div}{\bf v}]\eta dz+\int_{\Gamma_{in}}(\bar U +\bar{\bf v}^{(1)})w_{in}\eta ds
\end{aligned}
\end{equation}
where $\eta \in\tilde C^{\infty}(\Omega)$ means that $\eta\in C^{\infty}(\Omega)$ and $\eta|_{\Gamma_{out}}=0$.
\end{definition}
we want to apply the Galerkin method to prove the existence of weak solution. To 
this end, introduce an orthonormal basis of $\{\omega_k\}\subset V$ and finite 
dimensional subspace $V^N=\{\Sigma_{i=1}^Nc_i\omega_i:c_i\in \mathbb{R}\}$. We look 
for the approximate velocity field of the form ${\bf v}^N=\Sigma_{i=1}^N 
c_i\omega_i$.
Note that the solution of the continuity equation has been given by the operator $S$ defined in (\ref{def-of-S}).Now we proceed with the Galerkin scheme.Taking $g=g^N, {\bf v}={\bf v}^N=\Sigma_{i=1}^N c_i\omega_i, \phi=\omega_k, k=1,2...N$ and $w=w^N=S(f^N-\divg {\bf v}^N)$, where $f^N,g^N$ are orthogonal projections of $f,g$ on $V^N$. We arrive at a system of N equations
\begin{equation}\label{app-sys}
  B^N({\bf v}^N,\omega_k)=0,\quad k=1,2...N,
\end{equation}
where $B^N:V^N\times V^N\rightarrow \mathbb{R}$ is defined as
\begin{equation*}
\begin{aligned}
    B^N({\bf v}^N,\phi)&=\int_{\Omega}\{{\bf v}^N\bar U\partial_{z_1}\phi+({\bf 
    v}^N\cdot\nabla\bar U_0)\cdot\phi+2\mu D({\bf v}^N):\nabla\phi+\nu\divg {\bf 
    v}^N\divg \phi\}dz\\
    &-\int_{\Omega}g^N\cdot \phi dz-\gamma\int_{\Omega}S(f^N-\divg {\bf v}^N)\divg 
    \phi dz\\
    &+\int_{\Gamma_0}[\alpha({\bf v}^N\cdot\tau)-\tilde B](\phi\cdot \tau)d\sigma.
\end{aligned}
\end{equation*}
Now, if ${\bf v}^N$ satisfies (\ref{app-sys}) for $k=1,2...N$, then the pair $({\bf v}^N,w^N)$ satisfies (\ref{Lin-sys}) for $(\phi,\eta)\in V\times \tilde C^{\infty}$ with $w=w^N$. we call such a pair an approximation solution.  In order to show the existence of approximation solution, we apply the following results for finite dimensional Hilbert space\cite{NS2004}.
\begin{lemma}
  Let X be a finite dimensional Hilbert space and let $P:X\rightarrow X$ be a continuous operator satisfying
  \begin{equation*}
    \exists M>0:\,(p(\xi),\xi)>0\,\,for\,\,\|\xi\|=M,
  \end{equation*}
  then there is at leats one $\xi^\star$ such that $\|\xi^\star\|\leq M$ and $P(\xi^\star)=0$.
\end{lemma}
The existence of approximation solution is guaranteed by the following lemma.
\begin{lemma}\label{fixpt-lem}
Let $f,g\in L^2(\Omega)$,\,$w_{in}\in L^2(\Gamma_{in}),\tilde B\in L^2(\Gamma_0)$. If $\|\bar{\bf v}\|_{W^{2,p}}$ is small enough and the viscous coefficients $\mu>\frac{1}{2\pi^2}$, then there exists ${\bf v}^{N}\in V^{N}$ fulfills (\ref{app-sys}). Furthermore, there exists a positive constant M indecent of N such that
\begin{equation}\label{uni-est-app}
\|{\bf v}^{N}\|_{H^1}\leq M.
\end{equation}
\end{lemma}
\begin{proof}
Define $P^N:V^N\rightarrow V^N$ as
\begin{equation}
  P^N(\xi^N)=\Sigma_{k=1}^NB^N(\xi^N,\omega_k)\omega_k\,\,\text{for}\,\,\xi^N\in V^N.
\end{equation}
According to Lemma\, \ref{fixpt-lem}, we need to show that $(p(\xi^N),\xi^N)>0$ on some sphere in $V^N$. Since $B^N(\cdot,\cdot)$ is linear with respect to the second variable, we have
\begin{equation}\label{bili-est}
\begin{aligned}
(p^N(\xi^N),\xi^N)=B^N(\xi^N,\xi^N)&=2\mu\int_{\Omega} |D(\xi^N)|^2dz+\nu\int_{\Omega}|\divg\xi^N|^2dz\\
&+\int_{\Omega} (\xi^N\cdot\nabla U_0)\cdot \xi^Ndz-\int_{\Omega} g^N\cdot \xi^N \\
&-\frac{1}{2}\int_{\Omega} \partial_{z_1}(\bar U\circ\psi_{\bar\xi^N})|\xi^N|^2 dz-\int_{\Gamma_{0}}\tilde B\xi^N\cdot\tau\\
&+\int_{\Gamma_{0}} \alpha|\xi^N|^2 ds-\gamma\int_{\Omega} S(f^N-\divg\xi^N)\divg\xi^Ndz.
\end{aligned}
\end{equation}
Due to Korn inequality similar as in the proof of Lemma \ref{H1-est}, we only need to find a bound on the lats term in  (\ref{bili-est}). Denote $\eta^N=S(f^N-\divg\xi^N)$, then
\begin{equation*}
  \int_\Omega\eta^N\divg\xi^Ndz=\int_\Omega f^N\eta^Ndz-\int_\Omega(\bar U+\bar{\bf v}^{(1)})\partial_{z_1}\eta^N\eta^Ndz,
\end{equation*}
the first term is controlled by
\begin{equation}
  \int_\Omega f^N\eta^Ndz\leq \|f\|_{L^2}\|\eta^N\|_{L^2}\leq\|f\|_{L^2}(\|f\|_{L^2}+\|\xi^N\|_{H^1}+|w_{in}|_{L^2})
\end{equation}
where we have used (\ref{estm-of-dsty}).
With the second integral we have
\begin{equation*}
  \begin{aligned}
    -\int_\Omega(\bar U+\bar{\bf v}^{(1)})\partial_{z_1}\eta^N\eta^Ndz&=\frac{1}{2}\int_{\Gamma_{in}}(\bar U+\bar{\bf v}^{(1)})(\eta^N)^2d\sigma-\frac{1}{2}\int_{\Gamma_{out}}(\bar U+\bar{\bf v}^{(1)})(\eta^N)^2d\sigma\\
    &+\frac{1}{2}\int_{\Omega}\partial_{z_1}(\bar U+\bar{\bf v}^{(1)})(\eta^N)^2dz\\
    &\leq C|w_{in}|^2_{L^2}+E\|\eta^N\|^2_{L^2}.
  \end{aligned}
\end{equation*}
Combine all those estimates together, we get
\begin{equation*}
  (p^N(\xi^N)\geq C(\|\xi^N\|^2_{H^1}-D\|\xi^N\|_{H^1}-D^2),
\end{equation*}
where $D=\|f\|_{L^2}+\|g\|_{V^{*}}+\|\tilde B\|_{L^2(\Gamma_{0})}+\|w_{in}\|_{L^2({\Gamma_{in}})}$. Hence there exists a constant $M$ such that $p^N(\xi^N)> 0$ for some $\|\xi^N\|=M$. Applying Lemma \ref{fixpt-lem} we conclude that there is a $\xi^\star$ such that $P^N(\xi^\star)=0$ and $\|\xi^\star\|\leq M.$ Finally, since $\{\omega_k\}$ is the basis of $V^N$, $P^N(\xi^\star)=0$ implies $B^N(\xi^\star,\omega_k)=0$, that is to say $\xi^\star$ is an approximation solution.
\end{proof}
Since the system is linear, the uniform estimate (\ref{uni-est-app}) immediately gives the existence of weak solution. The results is in the following.
\begin{lemma}
  Let $f,g\in L^2(\Omega)$,\,$w_{in}\in L^2(\Gamma_{in}),\tilde B\in L^2(\Gamma_0)$. If $\|\bar{\bf v}\|_{W^{2,p}}$ is small enough and the viscous coefficients $\mu>\frac{1}{2\pi^2}$, then the system (\ref{Lin-sys}) has a unique weak solution  $({\bf v},w)\in H^{1}\times L^{\infty}(L^2)$,  which satisfies the estimate (\ref{L2-estm}).
\begin{proof}
  The above lemma shows that the approximation solution ${\bf v}^N$ satisfy
\begin{equation*}
  \|{\bf v}^N\|_{H^1}\leq M
\end{equation*}
which combined with (\ref{estm-of-dsty}) gives
\begin{equation*}
  \|{\bf v}^N\|_{H^1}+\|w\|_{L^\infty (L^2)}\leq M.
\end{equation*}
Hence there exists a pair $({\bf v},w)\in H^1\times L^\infty(L^2)$ such that
\begin{equation*}
  {\bf v}^N\rightharpoonup{\bf v}\quad \text{in}\,H^1 ,
\end{equation*}
and
\begin{equation*}
  w^N\rightharpoonup w\quad \text{in}\,L^\infty (L^2).
\end{equation*}
 Since the system (\ref{Lin-sys}) is linear, passing to the limit in (\ref{def-weak-v}) for ${{\bf v}^N,w^N}$, it follows that ${\bf v}$ satisfy (\ref{def-weak-v}) with $w$. Similarly, taking the limit in $(\ref{def-weak-w})$ it is easy to verify $w=S(f-\divg{\bf v})$, the boundary condition on $w$ is guaranteed to hold by the definition of the operator $S$. It's obvious that $({\bf v},w)$ satisfies the estimate (\ref{L2-estm}). The proof is thus complete.
\end{proof}
\end{lemma}
\subsection{Strong solution}
In this section, we will show that the weak solution given above is also strong by using symmetric extension methods.
\begin{lemma}
Let $f\in W^{1,p}(\Omega),g\in L^p(\Omega),w_{in}\in W^{1,p}(\Gamma_{in})\tilde B\in 
W^{1-\frac{1}{p},p}(\Gamma_0)$. If $\|\bar{\bf v}\|_{W^{2,p}}$ is small enough and 
the  viscous coefficients $\mu>\frac{1}{2 \pi^2}$ , then the system (\ref{Lin-sys}) 
has a unique strong solution  $({\bf v},w)\in W^{2,p}\times W^{1,p}$,  which 
satisfies the estimate (\ref{final-w2p}).
\end{lemma}
\begin{proof}
Since (\ref{Lin-sys}) is a linear system, the a priori estimates (\ref{final-w2p}) will deduce the regularity of the weak solution in the interior. In order deal with the singularity of the boundary at the junctions of $\Gamma_0$ with $\Gamma_{in}$ and $\Gamma_{out}$, we still apply the symmetric extension methods introduced in the proof of lemma (\ref{diff-quo}). Hence we can extend the weak solution on the negative values of $z_2$, and using the estimates (\ref{final-w2p}) show that the extended solution has the same regularity.
\end{proof}
\section{proof of main results}
According to the analysis in Section 4. We can define the solution operator of linear systems as
\begin{equation}
T:W^{2,p}(\Omega)\times W^{1,p}(\Omega)\rightarrow W^{2,p}(\Omega)\times W^{1,p}(\Omega)
\end{equation}
such that $({\bf v}^{n+1},w^{n+1})=T({\bf v}^n,w^n)$ is the solution of the linear system (\ref{Itrn-sys}). In the following we want to show that T is a contraction mapping in some proper subspace, hence $T$ has a unique solution of (\ref{Z-cordin-sys}).
\begin{lemma}
There is a positive constant $R>0$ such that the operator $T$ mapping $B_R$ into itself provided $D_0$ is small enough, where $$B_R=\{({\bf v},w)\in W^{2,p}(\Omega)\times W^{1,p}:\|{\bf v}\|_{W^{2,p}(\Omega)}+\|w\|_{W^{1,p}(\Omega)}\leq R \}.$$
\end{lemma}
\begin{proof}
By (\ref{est-soc-z}) and (\ref{final-w2p}) we have
\begin{equation}
\begin{aligned}
\|T&({\bf v}^{n},w^n)\|_{W^{2,p}\times W^{1,p}}=\|({\bf v}^{n+1},w^{n+1})\|_{{W^{2,p}\times W^{1,p}}}\\
&\leq C[(\|{\bf  v}^{n}\|_{2,p}+\|w^{n}\|_{1,p})^3+(\|{\bf  v}^{n}\|_{2,p}+\|w^{n}\|_{1,p})^2]\\
&+E(\|{\bf  v}\|_{W^{2,p}}+\|w\|_{W^{1,p}})+\|\tilde{\bf u}\|_{W^{2,p}}\\
&\leq C(R^3 + R^2)+E(D_0)
\end{aligned}
\end{equation}
for $R<1$, and hence for $R\leq 2E(D_0)$ with $E(D_0)\leq\frac{1}{16C}$, we have
$$\|({\bf v}^{n+1},w^{n+1})\|_{{W^{2,p}\times W^{1,p}}}\leq 2E(D_0),$$
the proof thus is completed.
\end{proof}
\begin{remark}
The same proof also shows that $\|\partial_{z_1}w^{n+1}\|_{1;p}\leq E(D_0)$,\, 
thanks to the second term in the r.h.s of (\ref{final-w2p}),
\end{remark}
In the next lemma we will show that $T$ is contraction.
\begin{lemma}
There is a constant $0<E_0<1$ independent of n,such that
\begin{equation}\label{diff-est}
\|{\bf v}^{n+1}-{\bf v}^{m+1}\|_{{2,p}:\Omega}+\|{w}^{n+1}-{w}^{m+1}\|_{{1,p}:\Omega}\leq E_0[ \|{\bf v}^{n}-{\bf v}^{m}\|_{{2,p}:\Omega}+\|{w}^{n}-{w}^{m}\|_{{1,p}:\Omega}]
\end{equation}
holds for all $({\bf v}^n,w^n)\in B_R$.
\end{lemma}
\begin{proof}
Making difference to derive the system of $({\bf v}^{n+1}-{\bf v}^{m+1},{w}^{n+1}-{w}^{m+1})$,\,we have
\begin{equation}\label{diff-sys}
\begin{aligned}
\ &((\bar U +({{\bf v}^{n} +\tilde{\bf u}})^{(1)})\circ\psi_{{\bf v}^n+\tilde{\bf u}})\partial_{z_1}(w^{n+1}-w^{m+1})+\operatorname{div}_{z} ({\bf v}^{n+1}-{\bf v}^{m+1})\\
 &=\tilde F({\bf v}^{n},w^n)-\tilde F({\bf v}^{m},w^m)  \\
&-\partial_{z_1}w^{m+1}[((\bar U +({{\bf v}^{n} +\tilde{\bf u}})^{(1)})\circ\psi_{{\bf v}^n+\tilde{\bf u}})-((\bar U +({{\bf v}^{m} +\tilde{\bf u}})^{(1)})\circ\psi_{{\bf v}^m+\tilde{\bf u}})],\\
\ &(\bar U\circ\psi_{{\bf v}^n+\tilde{\bf u}}) \partial_{z_1} ({\bf v}^{n+1}-{\bf v}^{m+1})+({\bf v}^{n+1}-{\bf v}^{m+1})\cdot \nabla_{z} U_0+\gamma \nabla_{z} (w^{n+1}-w^{m+1})\\
&-\mu \Delta_{z} ({\bf v}^{n+1}-{\bf v}^{m+1})-(\mu+\nu) \nabla \operatorname{div}_{z}({\bf v}^{n+1}-{\bf v}^{m+1}) =\tilde G({\bf v}^n,w^n)-\tilde G({\bf v}^m,w^m)  \\
&-{\bf v}^{m+1} (\bar U\circ\psi_{{\bf v}^n+\tilde{\bf u}}-\bar U\circ\psi_{{\bf v}^m+\tilde{\bf u}}),\\
\ & {\bf v}^{n+1}- {\bf v}^{m+1} =0,\\
\ &w^{n+1}-w^{m+1} =0,\\
\ &({\bf v}^{n+1}-{\bf v}^{m+1}) \cdot \mathbf{n} =0, \\
& 2\mu\mathbf{n}\cdot D_z({\bf v}^{n+1}-{\bf v}^{m+1})\cdot \tau +\alpha({\bf v}^{n+1}-{\bf v}^{m+1})\cdot \tau =2\mu{\bf n}\cdot[R({v}^{m},D)-R({\bf v}^{n},D)]\cdot\tau.
\end{aligned}
\end{equation}
Applying $W^{2,p}$ estimates to the $(\ref{diff-sys})$,we obtain
\begin{equation}
\begin{aligned}
&\|{\bf v}^{n+1}-{\bf v}^{m+1}\|_{{2,p};\Omega}+\|w^{n+1}-w^{m+1}\|_{{1,p};\Omega}\\
&\leq C\{\|\partial_{z_1}w^{m+1}[((\bar U +({{\bf v}^{n} +\tilde{\bf u}})^{(1)})\circ\psi_{{\bf v}^n+\tilde{\bf u}})-((\bar U +({{\bf v}^{m} +\tilde{\bf u}})^{(1)})\circ\psi_{{\bf v}^m+\tilde{\bf u}})]\|_{1,p}\\
&+\|\tilde F({\bf v}^{n},w^n)-\tilde F({\bf v}^{m},w^m)\|_{1,p}+\|{\bf v}^{m+1} (\bar U\circ\psi_{{\bf v}^n+\tilde{\bf u}}-\bar U\circ\psi_{{\bf v}^m+\tilde{\bf u}})\|_{p}\\
&+\|\tilde G({\bf v}^n,w^n)-\tilde G({\bf v}^m,w^m)\|_{p}+\|2\mu{\bf n}\cdot[R({\bf v}^{m},D)-R({\bf v}^{n},D)]\cdot\tau\|_{1-\frac{1}{p},p;\Gamma_{0}}\}\\
&:=I_1+I_2+I_3+I_4+I_5
\end{aligned}
\end{equation}
Estimate of $I_1$.\, By the definition of $\psi,$ it  continuously dependent on 
the parameters ${\bf v}$,\, which implies that
\begin{equation}
|\psi_{{\bf v}^n+\tilde{\bf u}}-\psi_{{\bf v}^m+\tilde{\bf u}}|\leq C|{\bf v}^{n}-{\bf v}^{m}|
\end{equation}
for some positive constant $C$.\, Hence,
\begin{equation}
\begin{aligned}
|({{\bf v}^{n} +\tilde{\bf u}})^{(1)}\circ\psi_{{\bf v}^n+\tilde{\bf u}}-({{\bf v}^{m} +\tilde{\bf u}})^{(1)}\circ\psi_{{\bf v}^m+\tilde{\bf u}}|\\
\leq C[|\nabla({\bf v}^n+\tilde{\bf u})^{(1)}||{\bf v}^{n}-{\bf v}^{m}|+|{\bf v}^{n}-{\bf v}^{m}|]
\end{aligned}
\end{equation}
Consequently,
\begin{equation}
\begin{aligned}
I_1&\leq \|\partial_{z_1}w^{m+1}\|_{1,p}\|{\bf v}^{n}-{\bf v}^{m}\|_{2,p}\\
&\leq E(D_0)\|{\bf v}^{n}-{\bf v}^{m}\|_{2,p}
\end{aligned}
\end{equation}
 It's easy to show that the same estimates holds for $I_3$.\\
Estimates of $I_2$.\, Note that (\ref{deriv-in-x-z}) implies that
\begin{equation} 
\begin{aligned}
|R(\cdot,\partial)|\leq E\|\cdot\|_{1,p},\\
|R(\cdot,\partial^2)|\leq E\|\cdot\|_{2,p}.\\
\end{aligned}
\end{equation}
Hence, to show the bounds of $I_2$,\, it sufficient to estimates $\|F({\bf v}^n,w^n)-F({\bf v}^m,w^m)\|_{1,p}$.
\begin{equation}
  \begin{aligned}
   F({\bf v}^n,w^n)-F({\bf v}^m,w^m)&=w^n\divg ({\bf v}^n+{\bf\tilde u})-w^m\divg ({\bf v}^m+{\bf\tilde u})\\
   &=w^n\divg{\bf v}^m-w^m\divg{\bf v}^m+(w^n-w^m)\divg{\bf\tilde u},
  \end{aligned}
\end{equation}
and
\begin{equation}
  \begin{aligned}
   &w^n\divg{\bf v}^n-w^m\divg{\bf v}^m\\
   &=w^n[\divg_z{\bf v}^n+R({\bf v}^n,\divg)]-w^m[\divg_z{\bf v}^m+R({\bf v}^m,\divg)]\\
   &=(w^n-w^m)\divg{\bf v}^n+w^m\divg({\bf v}^n-{\bf v}^m)\\
   &=(w^n-w^m)R({\bf v}^n,\divg)+w^mR({\bf v}^n-{\bf v}^m,\divg),
  \end{aligned}
\end{equation}
then it is immediate that
\begin{equation}
  \begin{aligned}
   \|F({\bf v}^n,w^n)-F({\bf v}^m,w^m)\|_{1,p}\leq C(D_0)(\|w^n-w^m\|_{1,p}+\|{\bf v}^{n}-{\bf v}^{m}\|_{2,p}).
  \end{aligned}
\end{equation}
Estimates of $I_4$.\, Due to the same reason, we just need to estimates 
$\|G({\bf v}^n,w^n)-G({\bf v}^m,w^m)\|_{p}$.\, For the convenience, denote 
$\delta P^{\prime}(w)=P^{\prime}(w+1)-P^{\prime}(w)$, then
\begin{equation}
\begin{aligned}
&\delta P^{\prime}(w^n)\nabla_x w^n-\delta P^{\prime}(w^m)\nabla_x w^m\\
&=\delta P^{\prime}(w^n)\nabla_x (w^n-w^m)-\nabla_x w^m(\delta P^{\prime}(w^m)-\delta P^{\prime}(w^n))\\
&=\delta P^{\prime}(w^n)[\nabla_z (w^n-w^m)+R(w^n-w^m,\nabla)]\\
&+[\nabla_z w^m+R(w^m,\nabla)](\delta P^{\prime}(w^m)-\delta P^{\prime}(w^n))
\end{aligned}
\end{equation}
which gives
\begin{equation}
\|\delta P^{\prime}(w^n)\nabla_x w^n-\delta P^{\prime}(w^m)\nabla_x w^m\|_{p}\leq E(\|w^n\|_{1,p},\|w^m\|_{1,p})\|w^n-w^m\|_{1,p}.
\end{equation}
The estimates of reminder terms are similar, so we show one of them as a example.
\begin{equation}
\begin{aligned}
&w^{n}{\bf v}^{n}\cdot\nabla_x {\bf v}^{n}-w^{m}{\bf v}^{m}\cdot\nabla_x {\bf v}^{m}\\
&=w^{n}{\bf v}^{n}\cdot[\nabla_z {\bf v}^{n}+R({\bf v}^n,\nabla)]-w^{m}{\bf v}^{m}\cdot[\nabla_z {\bf v}^{m}+R({\bf v}^m,\nabla)]\\
&=(w^n-w^m){\bf v}^{n}\cdot\nabla_z{\bf v}^n+w^m({\bf v}^n-{\bf v}^m)\cdot\nabla_z{\bf v}^n+w^m{\bf v}^m\cdot\nabla_z ({\bf v}^n-{\bf v}^m),  \\
&+(w^n-w^m){\bf v}^{n}\cdot R({\bf v}^n,\nabla)+w^m({\bf v}^n-{\bf v}^m)\cdot R({\bf v}^n,\nabla)+w^m{\bf v}^m R({\bf v}^n-{\bf v}^m,\nabla),
                                                                                                                                                                                                                                \end{aligned}
\end{equation}
then, we have
\begin{equation}
  \begin{aligned}
    \|G({\bf v}^n,w^n)-G({\bf v}^m,w^m)\|_{p}\leq C(D_0)(\|w^n-w^m\|_{1,p}+\|{\bf v}^{n}-{\bf v}^{m}\|_{2,p}).
  \end{aligned}
\end{equation}
Estimates of $I_5$.\,
\begin{equation}
\begin{aligned}
\|&2\mu{\bf n}\cdot[R({\bf v}^{m},D)-R({\bf v}^{n},D)]\cdot\tau\|_{1-\frac{1}{p},p;\Gamma_{0}}\\
&\leq\|R({\bf v}^{m}-{\bf v}^n,D)\|_{1,p}\\
&\leq C(E)\|{\bf v}^{n}-{\bf v}^{m}\|_{2,p}.
\end{aligned}
\end{equation}
Due to the smallness assumption of $C(\cdot)$, one can choose a small constant $0<E_0<1$ independent of $n$ fulfills  (\ref{diff-est}).
\end{proof}
{\bf Proof of Theorem 1.}we have prove that the operator $T$ is contraction mapping on $B_R\in W^{2,p}\times W^{1,p}$ for some small $R=R(D_0)$. Hence the Banach fixed point theorem gives existence of a unique fixed point in the ball $B_R$. By the definition of $T$, the fixed point is exactly the solution of the system (\ref{Z-cordin-sys}). Finally, we can change variable to the original coordinate, and in $x$ variable our solution satisfies the system (\ref{Lin-Homo-sys}). The proof is thus complete.

\bibliographystyle{amsplain}

\begin{thebibliography}{000}
\bibitem{1979B}M.E. Bogovskii, Solution of the first boudary value problem for the equation of continuity of an incompressible medium, Dokl. Akad. Nauk SSSR, 248 (1979) 1037–1040 (in Russian).

\bibitem{2000CJ}H. Choe and B. Jin, Existence of solutions of stationary compressible Navier-Stokes equations with large force, J. Funct. Anal. 177 (2000)  54–88.

\bibitem{2015DJJY}C. Dou, F. Jiang, S. Jiang, Y.-F. Yang, Existence of strong solutions to the steady Navier–Stokes equations for a compressible heat-conductive fluid with large forces. J. Math. Pures Appl. 103 (2015) 1163–1197.

\bibitem{2001F}E. Feireisl, On compactness of solutions to the compressible isentropic Navier–Stokes equations when the density is not square integrable. Comment.Math. Univ. Carolinae, 42 (2001) 83–98.

\bibitem{2004F} E. Feireisl, Dynamics of Viscous Compressible Fluids, Oxford Lecture Ser. Math. Appl., vol. 26, Oxford University Press, Oxford, 2004.

\bibitem{2018FN}E. Feireisl,A. Novotný, Stationary solutions to the compressible Navier-Stokes system with general boundary conditions. Ann. Inst. H. Poincaré Anal. Non Linéaire 35 (2018) 1457–1475.

\bibitem{JMW2010} J.Frehse, M.Steinhauer, W.Weigant, The Dirichlet problem for viscous compressible isothermal Navier-Stokes equations in two dimensions, Arch. Ration. Mech. Anal. 198 (2010), 1–12.

\bibitem{1994H} Ito, Hiroya, Best constants in Korn-Poincaré's inequalities on a slab, Math. Methods Appl. Sci. 17 (1994) 525–549.

\bibitem{GJZ2014}Guo Yan, Jiang Song, Zhou Chunhui, Steady viscous compressible channel flows. SIAM J. Math. Anal. 47 (2015) 3648–3670.

\bibitem{2011JK}C.H. Jun and J.R. Kweon, For the stationary compressible viscous Navier–Stokes equations with no-slip condition on a convex polygon, J. Differential Equations, 250 (2011) 2440–2461.

\bibitem{KK1997}R.B. Kellogg, J.R. Kweon, Compressible Navier–Stokes equations in a bounded domain with inflow boundary condition, SIAM J. Math. Anal. 28  (1997), 94–108.

\bibitem{KK2002}J.R. Kweon and R.B. Kellogg, Regularity of solutions to the Navier-Stokes equations for compressible barotropic flows on a polygon, Arch. Ration. Mech. Anal. 163(2002), 35–64.

\bibitem{2004K}J.R. Kweon, A regularity result of solution to the compressible Stokes equations on a convex polygon. Z. Angew. Math. Phys. 55 (2004) 435–450.

\bibitem{KS2006}J.R. Kweon, M. Song, Boundary geometry and regularity of solution to the compressible Navier–Stokes equations in bounded domains of $\mathbb{R}^n$, ZAMM Z. Angew. Math. Mech. 86 (2006), 495–504.

\bibitem{1998L} P.L. Lions, Mathematical Topics in Fluid Mechanics, vol. 2: Compressible Models, Oxford Lecture Ser. Math. Appl., vol. 10,
Oxford Science Publications, The Clarendon Press, Oxford University Press, New York, 1998.

\bibitem{2003MP}Piotr B. Mucha, Piotr Bogusław, On Navier-Stokes equations with slip boundary conditions in an infinite pipe. Acta Appl. Math. 76 (2003)  1–15.

\bibitem{MP2014} Piotr B. Mucha, T. Piasecki, Compressible perturbation of Poiseuille type flow, J. Math. Pures Appl. 102 (2014) 338–363.

\bibitem{1996NNK}S.A. Nazarov, A. Novotny, and K. Pileckas, On steady compressible Navier–Stokes equations in plane domains with corners, Math. Ann. 304 (1996) 121–150.

\bibitem{NN2002}
Novo, S. and Novotný, A., On the existence of weak solutions to steady compressible 
Navier–Stokes equations when the density is not square integrable, J. Math. Kyoto 
Univ. 42 (2002) 531–550.

\bibitem{NS2004} A. Novotny, I. Straskraba, Introduction to the Mathematical Theory of Compressible Flow, Oxford Lecture Series in MAthematics and its Applications 27.Oxford University Press Inc.,New York, 2004.

\bibitem{2018OJ}Kwon Ohsung, Kweon Jae Ryong, Compressible Navier-Stokes equations in a polyhedral cylinder with inflow boundary condition, J. Math. Fluid Mech. 20 (2018)  581–601.

\bibitem{P2009} T. Piasecki, Steady compressible Navier–Stokes flow in a square, J. Math. Anal. Appl. 357 (2009) 447–467.

\bibitem{P2010} T. Piasecki,On an inhomogeneous slip-inflow boundary value problem for a steady flow of a viscous compressible fluid in a cylindrical domain, J. Differential Equations 248 (2010) 2171–2198.

\bibitem{2014PP}T. Piasecki, M. Pokorný, Strong solutions to the Navier–Stokes–Fourier system with slip-inflow boundary conditions, Z. Angew. Math. Mech. 94 (2014) 1035–1057.

\bibitem{PW2015}P. I. Plotnikov,W. Weigant, Steady 3D viscous compressible flows with adiabatic exponent $\gamma\in$(1,$\infty$), J. Math. Pures Appl. 104 (2015), 58–82.

\bibitem{PRS2008}P. I. Plotnikov, E.V. Ruban, and J. Sokolowski, Inhomogeneous boundary value problems for compressible Navier–Stokes Equations: Well-posedness and sensitivity analysis, SIAM J. Math. Anal. 40 (2008), 1152–1200.

\bibitem{JZ2011}Jiang Song, Zhou Chunhui, Existence of weak solutions to the three-dimensional steady compressible Navier-Stokes equations, Ann. Inst. H. Poincaré Anal. Non Linéaire 28 (2011), 485–498.

\bibitem{1986VZ}A. Valli and W.M. Zajaczkowski, Navier-Stokes equations for compressible fluids: Global existence and qualitative properties of the solutions in the general case, Comm. Math. Phys. 103 (1986) 259–296.

\bibitem{1987V}A. Valli, On the existence of stationary solutions to compressible Navier-Stokes equations, Ann. Inst. H. Poincar$\acute{e}$ Anal. Non Lin$\acute{e}$aire, 4 (1987)  99–113.

\end{thebibliography}

\end{document}